\documentclass[10pt,american,reqno]{amsart}
\usepackage{babel}
\usepackage{amsthm}
\usepackage{amsbsy}
\usepackage{amstext}
\usepackage{amsfonts}
\usepackage{mathcomp}
\usepackage{mathtools}
\usepackage{dsfont}
\usepackage{latexsym}
\usepackage{amssymb}
\usepackage{esint}
\usepackage{a4wide}
\usepackage[colorlinks,pdfpagelabels,pdfstartview = FitH,bookmarksopen = true,bookmarksnumbered = true,linkcolor = blue,plainpages = false,hypertexnames = false,citecolor = red,pagebackref=false]{hyperref}
\usepackage{cite}

\newtheorem{theo}{Theorem}[section]
\newtheorem{cor}[theo]{Corollary}

\newtheorem{lem}[theo]{Lemma}

\theoremstyle{definition}
\newtheorem{defin}[theo]{Definition}
\newtheorem*{lem*}{Lemma}
\newtheorem{rem}[theo]{Remark}

\newtheorem*{cor*}{Corollary}
\newtheorem*{theo*}{Theorem}
\newtheorem{example}[theo]{Example}

\renewcommand{\div}{\operatorname{div}}

\newcommand{\N}{\ensuremath{\mathbb{N}}}
\newcommand{\R}{\ensuremath{\mathbb{R}}}

\def\Xint#1{\mathchoice
    {\XXint\displaystyle\textstyle{#1}}%
    {\XXint\textstyle\scriptstyle{#1}}%
    {\XXint\scriptstyle\scriptscriptstyle{#1}}%
    {\XXint\scriptscriptstyle\scriptscriptstyle{#1}}%
    \!\int}
\def\XXint#1#2#3{\setbox0=\hbox{$#1{#2#3}{\int}$}
    \vcenter{\hbox{$#2#3$}}\kern-0.5\wd0}
\def\bint{\Xint-}
\def\dashint{\Xint{\raise4pt\hbox to7pt{\hrulefill}}}

\def\XXiint#1#2#3{\setbox0=\hbox{$#1{#2#3}{\iint}$}
    \vcenter{\hbox{$#2#3$}}\kern-0.5\wd0}

\renewcommand{\epsilon}{\varepsilon}
\newcommand{\eps}{\varepsilon}
\renewcommand{\rho}{\varrho}
\newcommand{\ph}{\varphi}

\renewcommand{\epsilon}{\varepsilon}
\renewcommand{\rho}{\varrho}
\renewcommand{\d}{\:\! \mathrm{d}}

\DeclareMathOperator*{\esssup}{ess\,sup}
\DeclareMathOperator{\loc}{loc}

\numberwithin{equation}{section}
\allowdisplaybreaks

\begin{document}
\renewcommand{\refname}{References} 
\renewcommand{\abstractname}{Abstract} 
\title[Supercaloric functions]{Supercaloric functions for the parabolic $p$-Laplace equation in the fast diffusion case}
\author[R.Kr. Giri]{Ratan Kr. Giri}
\address{Ratan Kr. Giri\\
	Department of Mathematics, Aalto University\\
	P.~O.~Box 11100, FI-00076 Aalto University, Finland}
\email{giri90ratan@gmail.com}
\author[J. Kinnunen]{Juha Kinnunen}
\address{Juha Kinnunen\\
	Department of Mathematics, Aalto University\\
	P.~O.~Box 11100, FI-00076 Aalto University, Finland}
\email{juha.k.kinnunen@aalto.fi}
\author[K. Moring]{Kristian Moring}
\address{Kristian Moring\\
	Department of Mathematics, Aalto University\\
	P.~O.~Box 11100, FI-00076 Aalto University, Finland}
\email{kristian.moring@aalto.fi}

\begin{abstract}
We study a generalized class of supersolutions, so-called $p$-supercaloric functions, to the parabolic $p$-Laplace equation.
This class of functions is defined as lower semicontinuous functions that are finite in a dense set and satisfy the parabolic comparison principle. 
Their properties are relatively well understood for $p\ge 2$, but little is known in the fast diffusion case $1<p<2$. 
Every bounded $p$-supercaloric function belongs to the natural Sobolev space and is a weak supersolution to the parabolic $p$-Laplace equation for the entire range $1<p<\infty$. 
Our main result shows that unbounded $p$-supercaloric functions are divided into two mutually exclusive classes with sharp local integrability estimates for the function and its weak gradient in the supercritical case $\frac{2n}{n+1}<p<2$. 
The Barenblatt solution and the infinite point source solution show that both alternatives occur.
Barenblatt solutions do not exist in the subcritical case $1<p\le \frac{2n}{n+1}$ and the theory is not yet well understood.
\end{abstract}
\makeatother

\maketitle

\section{Introduction}

This paper  studies classes of supersolutions to the parabolic $p$-Laplace equation 
\begin{equation}\label{evo_eqn}
\partial_tu -\div\left( |\nabla u |^{p-2} \nabla u \right)=0.
\end{equation}  
The general theory covers the entire parameter range $1<p<\infty$, but different phenomena occur in the slow diffusion case $p>2$ and
in the fast diffusion case $1<p<2$. For $p=2$ we have the heat equation.
We do not only consider weak solutions, but also weak supersolutions and, more generally, $p$-supercaloric functions to \eqref{evo_eqn}. 
They are pointwise defined lower semicontinuous functions, finite in a dense subset, and are required to satisfy the comparison principle with respect to the solutions of \eqref{evo_eqn},
see Definition \ref{supercaloric} below. 
The definition of supercaloric functions is the same as in classical potential theory for the heat equation when $p=2$, see Watson \cite{W}. 
By Juutinen et al. \cite{juutinen}, the class of $p$-supercaloric functions is the same as the viscosity supersolutions to \eqref{evo_eqn} for $1<p<\infty$.
Our results can be extended to more general quasilinear equations
$$\partial_tu-\div A(x,t,u,\nabla u)=0,$$
with the $p$-growth,
which are discussed in DiBenedetto \cite{Di}, DiBenedetto et al. \cite{digi} and Wu et al. \cite{zhaobook}. 
For simplicity, we discuss only the prototype case in \eqref{evo_eqn}.

A $p$-supercaloric function does not, in general, belong to the natural Sobolev space for \eqref{evo_eqn}. 
The only connection to the equation is through the comparison principle. 
However, Kinnunen and Lindqvist\cite{KinnunenLindqvist2006} proved that bounded $p$-supercaloric functions belong to the appropriate Sobolev space  and are weak supersolutions to \eqref{evo_eqn} for $p \geq 2$. 
Korte et al.\cite{KoKuPa} extended the study for a more general class of parabolic equations with $p$-growth. 
In this paper we show that bounded $p$-supercaloric functions are weak solutions to \eqref{evo_eqn} for the entire range $1<p<\infty$.

We are mainly interested in unbounded $p$-supercaloric functions.
Assume that $u$ is a $p$-supercaloric function in $\Omega_T=\Omega\times(0,T)$, where $\Omega$ is an open set in $\R^n$ and $T>0$.
One of the main results of Kuusi et al. \cite{KuLiPa} asserts that for $p>2$ there are two mutually exclusive alternatives: 
Either $u\in L_{\loc}^q(\Omega_T)$ for every $0<q<p-1+\frac pn$ or $u\notin L_{\loc}^{p-2}(\Omega_T)$. 
In particular, if $u\in L_{\loc}^{p-2}(\Omega_T)$, then $u\in L_{\loc}^q(\Omega_T)$ for every $0<q<p-1+\frac pn$. 
For the corresponding theory for the porous medium equation, see Kinnunen et al.~\cite{kilelipa}.

Examples based on the Barenblatt solution (see Barenblatt~\cite{B}) and the friendly giant show that both alternatives occur.
In the first alternative the upper bound for the exponent is given by the Barenblatt solution
\begin{equation}\label{e.deg_barenblatt}
U(x,t)=
(\lambda t)^{-\frac{n}{\lambda}}\left(c-\tfrac{p-2}{p}(\lambda t)^{-\tfrac{p}{\lambda(p-1)}}|x|^{\frac{p}{p-1}}\right)_+^{\frac{p-1}{p-2}},
\quad (x,t)\in\R^n\times(0,\infty),
\end{equation}
where $2<p<\infty$, $\lambda = n(p-2)+p$ and the constant $c$ is a positive number, which can be chosen such that 
\begin{equation}\label{e.normalization}
\int_{\R^n}U(x,t)\,\d x=1
\end{equation} 
for every $t>0$. 
The Barenblatt solution is a weak solution to \eqref{evo_eqn} in $\R^n\times(0,\infty)$ and the zero extension
\begin{equation}\label{e.zero_barenblatt}
u(x,t)=
\begin{cases}
U(x,t),&\quad t >0,\\
0, &\quad t \leq 0,
\end{cases}
\end{equation}
is $p$-supercaloric in $\R^n \times \R$ with $u\in L_{\loc}^q(\R^n \times \R)$ for every $0<q<p-1+\frac pn$.
This function solves \eqref{evo_eqn} with a finite point source, but it fails to belong to the natural Sobolev space, since $|\nabla u|\notin L_{\loc}^p(\R^n \times \R)$.
For the second alternative, we consider a bounded open set $\Omega$ in $\R^n$ with a smooth boundary.
By separation of variables, we obtain the friendly giant 
\[
U(x,t) =t^{-\frac{1}{p-2}}u(x),
\]
where $u\in C(\Omega)\cap W^{1,p}_{0}(\Omega)$ is a weak solution to the elliptic equation 
$$
\div\bigl(|\nabla u|^{p-2}\nabla u\bigr)+\tfrac{1}{p-2}u=0
$$
in $\Omega$ with $u(x)>0$ for every $x\in\Omega$.
The function $U$ is a weak solution to \eqref{evo_eqn} in $\Omega\times(0,\infty)$ and the zero extension as in \eqref{e.zero_barenblatt} 
is $p$-supercaloric in $\Omega \times \R$ with $u\notin L_{\loc}^{p-2}(\Omega \times \R)$.

We prove the corresponding result in the supercritical range $\frac{2n}{n+1}<p<2$ for a $p$-supercaloric function $u$ in $\Omega_T$.
It asserts that either $u\in L_{\loc}^q(\Omega_T)$ for every $0<q<p-1+\frac pn$ or $u\notin L^{\frac{n}{p}(2-p)}_{\loc}(\Omega_T)$. 
In particular, if $u\in L_{\loc}^{\frac{n}{p}(2-p)}(\Omega_T)$, then $u\in L_{\loc}^q(\Omega_T)$ for every $0<q<p-1+\frac pn$.
Again, both alternatives occur.
For $\frac{2n}{n+1}<p<2$, the Barenblatt solution (see Wu et al. \cite[2.7.2]{zhaobook} and Bidaut-V\'{e}ron \cite{veron}) of $\eqref{evo_eqn}$ is given by formula
\begin{equation}\label{e.sing_barenblatt}
U(x,t)=
(\lambda t)^{-\frac{n}{\lambda}}\left(c+ \tfrac{2-p}{p}(\lambda t)^{-\frac{p}{\lambda(p-1)}}|x|^{\frac{p}{p-1}}\right)^{-\frac{p-1}{2-p}},
\quad (x,t)\in\R^n\times(0,\infty),
\end{equation}
where $\lambda = n(p-2)+p$ and the constant $c$ is a positive number so that \eqref{e.normalization} holds for every $t>0$. 
Observe that $p >\frac{2n}{n+1}$ is equivalent with $\lambda > 0$. 
Barenblatt solutions do not exist in the subcritical case $1<p\le \frac{2n}{n+1}$ and, to our knowledge, the theory is not yet well understood.
The Barenblatt solution is compactly supported for every $t>0$ for $2<p<\infty$, but it is positive everywhere for $\frac{2n}{n+1}<p<2$.
As above, the zero extension to the negative times is $p$-supercaloric in $\R^n \times \R$ with $u\in L_{\loc}^q(\R^n \times \R)$ for every $0<q<p-1+\frac pn$.
A prime example of a $p$-supercaloric function for $\frac{2n}{n+1}<p<2$ that does not belong to the Barenblatt class is
the infinite point source solution
\begin{equation} \label{razorblade}
U(x,t)=\left(\frac{ct}{|x|^p}\right)^{\frac{1}{2-p}}, \quad (x,t)\in\R^n\times(0,\infty),
\quad c=(2-p)\left(\tfrac{p}{2-p}\right)^{p-1}\left(\tfrac{p}{2-p}-n\right),
\end{equation}
see Chasseigne and V\'azquez \cite{Vazquez}. 
This function is a solution to $\eqref{evo_eqn}$ in $(\R^n \setminus \{ 0 \})\times (0, \infty)$ and has singularity at $x=0$ for every $t > 0$.
Observe that \eqref{razorblade} is obtained by setting $c=0$ in the Barenblatt solution \eqref{e.sing_barenblatt} and thus solves \eqref{evo_eqn} with an infinite point source.
The zero extension as in \eqref{e.zero_barenblatt} is a $p$-supercaloric function $u$ in $\R^n \times \R$ with $u\notin L^{\frac{n}{p}(2-p)}_{\loc}(\R^n \times \R)$. 
Our main result asserts, roughly speaking, that a $p$-supercaloric function and its gradient have similar local integrability properties than the Barenblatt solution or the corresponding properties are at least as bad as for the infinite point source solution. 
A Moser type iteration scheme and Harnack estimates are applied in the argument.

{\bf Acknowledgments.} The authors would like to thank Peter Lindqvist for useful discussions and the Academy of Finland for support. K.~Moring has also been supported by the Magnus Ehrnrooth Foundation.

\section{Weak supersolutions and supercaloric functions}
We begin with notation. Let $\Omega \subset \R^n$ be an open set. 
For $T>0$ we denote a space-time cylinders in $\R^{n+1}$ by $\Omega_T=\Omega \times (0,T)$ and  $\Omega_{t_1, t_2}=\Omega\times (t_1,t_2)$, with $t_1<t_2$. 
The parabolic boundary of $\Omega_T$ is $\partial_p\Omega_T=(\overline\Omega\times\{0\})\cup(\partial\Omega\times(0,T])$.
We denote a cube in $\R^n$ by $Q=(a_1, b_1)\times \cdots \times (a_n, b_n)$ and the space-time cylinders of the form $Q_{t_1,t_2}$ with $t_1<t_2$, are called boxes.
$\Omega'_{t_1,t_2}\Subset\Omega_T$ denotes that $\overline{\Omega'_{t_1,t_2}}$ is a compact subset of $\Omega_T$.
 
Let $W^{1,p}(\Omega)$ denote the Sobolev space of functions $u\in L^p(\Omega)$, whose  first distributional partial derivatives $\frac{\partial u}{\partial x_i}$, $i=1,2,\dots, n$, exist in $\Omega$ and belong to $L^p(\Omega)$. The corresponding Sobolev space with zero boundary values is denoted by $W^{1,p}_0(\Omega)$. 
The parabolic Sobolev space $L^p(0,T;W^{1,p}(\Omega))$ consists of functions $u=u(x,t)$ such that for almost every $t \in (0,T)$, the function $x\mapsto u(x,t)$ belongs to $W^{1,p}(\Omega)$ and  
$$\iint_{\Omega_T} \big( |u(x,t)|^p+|\nabla u(x,t)|^p \big)\, \d x\, \d t<\infty.$$
Observe that the time derivative $\partial_t u$ does not appear anywhere. The definition of the space $L^p(0, T;W_0^{1,p}(\Omega))$ is similar. 
We denote $u\in L_{\loc}^p(0,T;W^{1,p}_{\loc}(\Omega))$ if $u\in L^p(t_1,t_2;W^{1,p}(\Omega'))$ for every $\Omega'_{t_1,t_2}\Subset\Omega_T$.
The definition for $L_{\loc}^p(\R;W^{1,p}_{\loc}(\R^n))$ is analogous.
Gradient  and divergence  are always taken with respect to the spatial variable only. 

Weak solutions to \eqref{evo_eqn} are assumed to belong to a parabolic Sobolev space, which guarantees a priori local integrability for the function and its weak gradient. 
We state the definition and results in a space-time cylinder $\Omega_T$, with $T>0$, but extensions to arbitrary cylinders $\Omega_{t_1,t_2}$, with $t_1<t_2$, are obvious. 

\begin{defin}
Let $1<p<\infty$ and let $\Omega$ be an open set in $\R^n$. A function $u\in L_{\loc}^p(0,T;W^{1,p}_{\loc}(\Omega))$ is called a weak solution to $\eqref{evo_eqn}$, if
\[
\iint_{\Omega_T} \left(-u \partial_t \varphi + |\nabla u|^{p-2}\nabla u \cdot \nabla \varphi \right)\,\d x\,\d t  =0
\]
for every $\varphi \in C^\infty_0(\Omega_T)$. 
Furthermore, we say that $u$ is a weak supersolution if the integral above is nonnegative for all nonnegative test functions $\varphi \in C^{\infty}_0(\Omega_T)$. If the integral is non-positive for such test functions, we call $u$ a weak subsolution.
\end{defin}

Locally bounded weak solutions are locally H\"older continuous, see \cite[Chapters III and IV]{Di}, and locally bounded gradients are locally 
H\"older continuous, see DiBenedetto \cite[Chapter IX]{Di}.
For lower semicontinuity of weak supersolutions, see Kuusi \cite{Kuusi2009}.
The statement of the following existence result can be found in Bj\"orn et al.~\cite[Theorem 2.3]{Anders}.
For the proof, we refer to Ivert \cite[Theorem 3.2]{ivert}, which applies Fontes \cite{Fontes2009}.
For related existence results, see B\"{o}gelein et al. \cite[Theorem 1.2]{BogeleinEtAl2014}, \cite[Theorem 1.2]{bogelein}.
Boundary behaviour in more general situations has been studied in Bj\"orn et al.~\cite{Anders} and Gianazza et al. \cite{GLL}.

\begin{theo}\label{t.existence}
Let $1<p<\infty$ and let $\Omega$ be a bounded open set in $\R^n$ with a Lipschitz boundary and $g\in C(\partial_p\Omega_T)$.
Then there exists a unique weak solution $u \in C(\overline{\Omega_T})$ to \eqref{evo_eqn}  
with $u=g$ on $\partial_p\Omega_T$. 
Moreover, if $g$ belongs to $C(0,T;L^2(\Omega))\cap L^p(0,T;W^{1,p}(\Omega))$, so does $u$. 
\end{theo}

We consider a comparison principle for weak super- and subsolutions. 
The argument is similar to Kilpel\"ainen and Lindqvist \cite[Lemma 3.1]{Kilpelainen} and Korte et al.\cite[Lemma 3.5]{KoKuPa}.

\begin{theo}\label{t.parcoparison}
Let $1<p<\infty$ and let $\Omega$ be a bounded open set in $\R^n$. Assume that $v$ is a weak supersolution and $u$ is a weak subsolution to $\eqref{evo_eqn}$ in $\Omega_T$. 
If $v$ and $-u$ are lower semicontinuous in $\overline{\Omega}_T$ and $u\leq v$ on $\partial_p\Omega_T$, then $u\leq v$ almost everywhere in $\Omega_T$.
\end{theo}

A pointwise limit of a sequence of uniformly bounded weak supersolutions is a weak supersolution, 
see Korte et al.\cite[Theorem 5.3]{KoKuPa}, see also Kinnunen and Lindqvist \cite{KinnunenLindqvist2006}.

\begin{theo}\label{converegnce_result}
	Let $1<p<\infty$ and let $\Omega$ be an open set in $\R^n$. Assume that $u_i$, $i=1,2,\dots$,  are weak supersolutions to $\eqref{evo_eqn}$  in $\Omega_T$ such that $\|u_i\|_{L^\infty(\Omega_T)}\leq L<\infty$
	 for every $i=1,2\dots$ and $u_i\to u$ almost everywhere in $\Omega_T$ as $i\to\infty$. Then $u$ is a weak supersolution to $\eqref{evo_eqn}$ in $\Omega_T$.
\end{theo}

The class of weak supersolutions is closed under taking minimum of two functions, see Korte el al. \cite[Lemma 3.2]{KoKuPa}. 
See also Kinnunen and Lindqvist \cite{KinnunenLindqvist2006}.

\begin{lem}\label{l.min}
Let $1<p<\infty$ and let $\Omega$ be an open set in $\R^n$.
If $u$ and $v$ are weak supersolutions to $\eqref{evo_eqn}$ in $\Omega_T$, then $\min\{u,v\}$ is a weak supersolution in $\Omega_T$. 
\end{lem}

The local boundedness assumption in Theorem \ref{converegnce_result} can be replaced with a uniform Sobolev space bound,
see Kinnunen and Lindqvist \cite{KinnunenLindqvist2006} and Korte et al.\cite[Remark 5.6]{KoKuPa}.
This implies that if $u\in L_{\loc}^p(0,T;W^{1,p}_{\loc}(\Omega))$ and $\min\{u,k\}\in L_{\loc}^p(0,T;W^{1,p}_{\loc}(\Omega))$ is a weak supersolution for every $k=1,2,\dots$, then $u$ is a weak supersolution.
In general, we have to consider more general class of solutions than weak supersolutions. 
Thus we define $p$-supercaloric functions as in Kilpel\"ainen and Lindqvist \cite[Lemma 3.1]{Kilpelainen}, see also  Kinnunen and Lindqvist \cite{KinnunenLindqvist2006}.
\begin{defin}\label{supercaloric}
Let $1<p<\infty$ and let $\Omega$ be an open set in $\R^n$.
A function $u \colon \Omega_T \to (-\infty,\infty]$ is called a $p$-supercaloric function, if
\begin{itemize}
\item[(i)] $u$ is lower semicontinuous,
\item[(ii)] $u$ is finite in a dense subset, and
\item[(iii)] $u$ satisfies the comparison principle on each box $Q_{t_1,t_2}\Subset \Omega_T$: if $h\in C(\overline{Q}_{t_1,t_2})$
 is a weak solution to $\eqref{evo_eqn}$ in $Q_{t_1,t_2}$ and if $h \leq u$ on the parabolic boundary of $Q_{t_1,t_2}$, then $ h\leq u$ in $Q_{t_1,t_2}$.
\end{itemize}
Similarly, an upper semicontinuous function $u$ is said to be a $p$-subcaloric function in $\Omega_T$ if $-u$ is $p$-supercaloric function in $\Omega_T$. 
\end{defin}

Since $p$-supercaloric functions are lower semicontinuous, they are locally bounded from below. 
Thus by adding a constant, we may assume that a $p$-supercaloric function is nonnegative, when we discuss local properties.
If $u$ is a nonnegative $p$-supercaloric function in $\Omega_T$, the zero extension in the past as in \eqref{e.zero_barenblatt}
is a $p$-supercaloric function in $\Omega \times (-\infty,T)$.
The following assertion is a direct consequence of the comparison principle in Definition \ref{supercaloric}. 

\begin{lem}\label{l.caloricmin}
Let $1<p<\infty$ and let $\Omega$ be an open set in $\R^n$.
If $u$ and $v$ are $p$-supercaloric functions in $\Omega_T$, then $\min\{u,v\}$ is a $p$-supercaloric function in $\Omega_T$. 
\end{lem}

Supercaloric functions are closed under increasing convergence, see \cite[Proposition 5.1]{KoKuPa}.

\begin{lem}\label{l.limit}
Let $1<p<\infty$ and let $\Omega$ be an open set in $\R^n$. Assume that $u_i$, $i=1,2,\dots$, are $p$-supercaloric functions in $\Omega_T$ such that $u_i\le u_{i+1}$ or every $i=1,2\dots$.
If $u=\lim_{i\to\infty} u$ is finite in a dense set, then $u$ is a $p$-supercaloric function in $\Omega_T$.
\end{lem}

A parabolic comparison principle holds for super- and subcaloric functions, see 
Kilpel\"ainen and Lindqvist \cite[Lemma 3.1]{Kilpelainen} and  Korte et al.~\cite[Lemma 3.5]{KoKuPa}. 

\begin{lem}
Let $1<p<\infty$ and let $\Omega$ be an open and bounded set in $\R^n$.
Assume that $u$ is a $p$-subcaloric function and $v$ is a $p$-supercaloric function in $\Omega_T$.
If $u\leq v$ on $\partial_p\Omega_T$, then $u\leq v$ in $\Omega_T$. 
\end{lem}

For the following comparison principle, we refer to
Bj\"orn et al.~\cite[Theorem 2.4]{Anders} and Korte et al.~\cite[Corollary 4.6]{KoKuPa}. 
Observe that the comparison is on the topological boundary instead of the parabolic boundary.

\begin{cor}\label{cor3}
Let $1<p<\infty$ and let $U$ be a bounded open set in $\R^{n+1}$.
Assume that $u$ is a $p$-supercaloric function in $U$ and $h\in C(\overline{U})$ is a weak solution to \eqref{evo_eqn} in $U$. 
If $h\leq u$ on $\{(x,t)\in\partial U:t<T\}$, then $h\leq u$ in $\{(x,t)\in U:t<T\}$. 
\end{cor}

We will apply an existence result for the obstacle problem. 
In order to guarantee continuity of the solution up to the boundary, we assume that the domain where the obstacle problem is considered has Lipschitz boundary.
This condition can be relaxed, see \cite[Theorem 3.1]{kokusi}.

\begin{theo}\label{obstacle problem}
Let $1<p<\infty$. Assume that $\Omega$ is an open and bounded set in $\R^n$ with Lipschitz boundary and let $\psi \in C(\overline{\Omega_T})$. 
	There exists  $u \in C(\overline{\Omega_T})$ that is a weak supersolution  to \eqref{evo_eqn} in $\Omega_T$ with the following properties:
	\begin{itemize}
		\item[(i)] $u\geq \psi$ in $\Omega_T$ and $u=\psi$ on $\partial_p\Omega_T$,
		\item[(ii)] $u$ is a weak solution in the open set $\{(x,t)\in\Omega_T:u(x,t)>\psi(x,t)\}$,
		\item[(iii)] $u$ is the smallest weak supersolution above $\psi$, that is, if $v$ is a weak supersolution in $\Omega_T$ and $v\ge\psi$ in $\Omega_T$, 
		then $v\ge u$ in $\Omega_T$.
	\end{itemize} 
\end{theo}
\begin{proof}
For $p>\frac{2n}{n+2}$, see Korte et al.\cite[Theorem 3.1]{kokusi}. 
A careful inspection of the proof reveals that the requirement $p>\frac{2n}{n+2}$ is used to obtain \cite[Theorem 2.7]{kokusi} 
and Theorem \ref{t.existence} in the space-time boxes with continuous boundary data which have been used in \cite[Construction 3.2]{kokusi}. 
A proof of  \cite[Theorem 2.7]{kokusi} for $p>\frac{2n}{n+2}$ is given in Korte et al. \cite[Theorem 5.3]{KoKuPa}, see also Kinnunen and Lindqvist \cite[Lemma 4.3]{KinnunenLindqvist2006} for $p>2$.
For the corresponding convergence result in the full range $1<p<\infty$ we refer to the Theorem $\ref{converegnce_result}$. 
\end{proof}

Next we show that every bounded $p$-supercaloric function $u$ is a weak supersolution in $\Omega_T$. 
In this case $u\in L_{\loc}^p(0,T;W^{1,p}_{\loc}(\Omega))$ and
\[
\iint_{\Omega_T} \left(-u \partial_t \varphi + |\nabla u|^{p-2}\nabla u \cdot \nabla \varphi \right)\,\d x\,\d t\ge 0
\]
for every nonnegative $\varphi \in C^\infty_0(\Omega_T)$. 
This result extends Kinnunen and Lindqvist \cite[Theorem 1.1]{KinnunenLindqvist2006} and Korte et al. \cite[Theorem 5.8]{KoKuPa}.
The argument is similar to  \cite[Theorem 5.8]{KoKuPa}, but we repeat it here to show that it applies in the full range $1<p<\infty$.

\begin{theo}\label{bdd_supersoln}
Let $1<p<\infty$ and let $\Omega$ be an open set in $\R^n$. 
If $u$ is a $p$-supercaloric function in $\Omega_T$ and $u$ is locally bounded, then $u$ is a weak supersolution to \eqref{evo_eqn} in $\Omega_T$.
\end{theo}

\begin{proof}
Since $u$ is lower semicontinuous, there exists a sequence of continuous functions $\psi_{i}$, $i=1,2,\dots$, 
such that $\psi_{1}\leq \psi_{2}\leq\dots\leq u$
and
\[
\lim_{i\to\infty} \psi_{i}(x,t)= u(x,t)
\]
for every $(x,t)\in\Omega_T$.
We claim that $u$ is a weak supersolution in $\Omega_T$. 
To conclude this it is sufficient to show that $u$ is a weak supersolution in every space-time box $Q_{t_1, t_2}\Subset\Omega_T$.

By Theorem $\ref{obstacle problem}$, for every $i=1,2,\dots$, there exists a solution $v_{i}\in C(\overline{Q_{t_1, t_2}})$ to the obstacle problem in $Q_{t_1, t_2}$ with the obstacle $\psi_{i}$. 
The function $v_{i}$ is a continuous weak solution in the open set 
\[
U=\{(x,t)\in Q_{t_1, t_2}: v_{i}(x,t)>\psi_{i}(x,t)\}\subset\R^{n+1}.
\]
Since $v_{i}=\psi_{i}$ on $\partial_p Q_{t_1, t_2}$, $v_{i}\in C(\overline{Q_{t_1, t_2}})$ and $\psi_{i}\in C(\overline{Q_{t_1, t_2}})$, it follows that $v_{i}=\psi_{i}$ on the boundary $\partial U$, except possibly when $t=t_2$. 
By Corollary~$\ref{cor3}$, we conclude that $u\geq v_{i}$ in $U$ for every $i=1,2,\dots$. 
Consequently, $\psi_{i}\leq v_{i}\leq u$ in $Q_{t_1, t_2}$ for every $i=1,2,\dots$ and thus
\[
\lim_{i\to\infty}v_{i}(x,t)= u(x,t)
\]
for every $(x,t)\in Q_{t_1, t_2}$.
According to the Theorem $\ref{converegnce_result}$, the function $u$ is a weak supersolution in $Q_{t_1, t_2}$.
Since this holds true for every $Q_{t_1, t_2}\Subset\Omega_T$ we conclude that $u$ is a weak supersolution in $\Omega_T$. 
\end{proof} 

\section{Barenblatt solutions}

This section discusses $p$-supercaloric functions with a Barenblatt type behaviour in the supercritical range.

\begin{example}
Let $\frac{2n}{n+1}<p<2$.
The Barenblatt solution $U(x,t)$ in \eqref{e.sing_barenblatt} is a weak solution to $\eqref{evo_eqn}$ in $\R^n\times(0,\infty)$. 
It follows that the zero extension $u(x,t)$ as in \eqref{e.zero_barenblatt}
is a $p$-supercaloric function in $\R^n\times\R$ according to Definition \ref{supercaloric} and it satisfies 
\[
\iint_{\R^n\times\R} \left(-u \partial_t \varphi + |\nabla u|^{p-2}\nabla u \cdot \nabla \varphi \right)\,\d x\,\d t
=M\varphi(0),
\]
with $M>0$, for every $\varphi \in C^\infty_0(\Omega_T)$. 
This means that $u$ solves the equation
$$
\partial_tu - \div(|\nabla u|^{p-2}\nabla u) = M \delta
$$
in the weak sense, where $\delta$ is Dirac's delta. 
It follows that the weak gradient $\nabla u(\cdot,t)$ exists, in the sense of~\eqref{e.limgrad}, for almost every $t$ and it is locally integrable to any power $0<q < p - 1 + \frac{1}{n+1}$. 
However,  the function $u$ is not a weak supersolution to \eqref{evo_eqn} in $\R^n\times\R$, since 
$$
\int_{t_1}^{t_2}\int_{B(0,r)} |\nabla u|^{p-1+\frac{1}{n+1}}\,\d x\,\d t=\infty
$$
for every $r>0$, $t_1\le0$ and $t_2>0$.
This implies that $u\notin L_{\loc}^p(\R;W^{1,p}_{\loc}(\R^n))$.
Observe, that the truncations $\min\{u,k\}$, $k=1,2,\dots$ belong to $L_{\loc}^p(\R;W^{1,p}_{\loc}(\R^n))$ and are weak solutions to \eqref{evo_eqn} in $\R^{n+1}$
by Theorem \ref{bdd_supersoln}.
\end{example}

Next we recall a Caccioppoli type inequality for nonnegative unbounded weak supersolutions, see Kuusi \cite[Lemma 2.2]{Kuusi2008}.

\begin{lem}\label{unbdd_caccio}
	Let $1<p<\infty$, $0<\eps<1$ and let $\Omega$ be an open set in $\R^n$. Assume that $u$ is a nonnegative weak supersolution in $\Omega_T$.
	There exists a constant $c = c(p,\eps)$ such that
	\begin{align*}
	\iint_{\Omega_T} &|\nabla u|^pu^{-\eps-1}\varphi^p\,\d x\, \d t + \esssup_{0<t<T} \int_\Omega u^{1-\eps} \ph^p\, \d x \\
	&\leq c \iint_{\Omega_T}u^{p-1-\eps} |\nabla \ph|^p  \, \d x\, \d t + c\iint_{\Omega_T}u^{1-\eps} |\partial_t (\ph^p)| \,\d x\, \d t
	\end{align*}
	for every nonnegative test function $\ph \in C_0^\infty(\Omega_T)$. 
\end{lem}

The following version of Sobolev's inequality will be useful for us, see DiBenedetto \cite[Proposition 3.1, p.\ 7]{Di} and DiBenedetto et al. \cite[Proposition 4.1]{digi}.

\begin{lem}\label{sobolev}
	Let $0<m<\infty$ and $1\le p<\infty$.
	Assume that $u\in L^p_{\loc}(0,T;W_{\loc}^{1,p}(\Omega))$ and $\varphi\in C^\infty_0(\Omega_T)$. Then there exists a constant $c=c(p,m,n)$ such that 
	\[
	\iint_{\Omega_T}|\varphi u|^q \d x\d t\leq c\iint_{\Omega_T}|\nabla(\varphi u)|^p\d x \d t\left(\esssup_{0<t<T}\int_\Omega|\varphi u|^m \d x\right)^{\frac{p}{n}},
	\]
	where $q=p+\frac{pm}{n}$.  
\end{lem}

We prove a general local integrability result for unbounded $p$-supercaloric functions. 
The obtained integrability exponent is sharp as shown by the Barenblatt solution.
We apply a Moser type iteration scheme. 

\begin{theo}\label{integbility_1}
	Let $\frac{2n}{n+1}<p<2$ and let $\Omega$ be an open set in $\R^n$. Assume that $u$ is a $p$-supercaloric function in $\Omega_T$. 
	If $ u \in L_{\loc}^{s}(\Omega_T) $	for some $s > \frac{n}{p}(2-p)$, then 
	$u\in L_{\loc}^q(\Omega_T)$ whenever $0<q<p-1+\frac{p}{n}$.
\end{theo}
\begin{proof}
Since $u$ is locally bounded from below, by adding a constant, we may assume that $u\ge1$. 
We may assume $0<s < 1$, otherwise we have $u\in L_{\loc}^1(\Omega_T)$ and we can directly proceed to the last step of the proof. 

Since  $p>\frac{2n}{n+1}$, we have $\frac{n}{p}(2-p)<1$. 
We consider the truncations $u_k=\min\{u, k\}$, $k=2, 3,\dots$. 
By Theorem \ref{bdd_supersoln}, the function $u_k$ is a weak supersolution to~\eqref{evo_eqn} and it satisfies the Caccioppoli estimate in Lemma $\ref{unbdd_caccio}$. Let $\varphi \in C_0^\infty(\Omega_T)$, $0\leq \varphi \leq 1$ and $\varphi=1$ in a compact subset of $\Omega_T$. 
For the first step of iteration, we choose $0<\epsilon<1$ in Lemma $\ref{unbdd_caccio}$ so that $1-\epsilon=s$. 
Since $u_k\ge1$, we obtain
\begin{align*}
\iint_{\Omega_T}\left|\nabla\left(\varphi u_k^{\frac{s-(2-p)}{p}}\right)\right|^p\,\d x\,\d t
& \leq c(p,s)\left(\iint_{\Omega_T}u_k^{s-(2-p)}|\nabla\varphi|^p\,\d x\,\d t
+ \iint_{\Omega_T} u_k^{s-2}|\nabla u_k|^p\varphi^p\,\d x\,\d t\right)\\
& \leq c(p,s)\left(\iint_{\Omega_T}u_k^{s-(2-p)}|\nabla\varphi|^p\,\d x\,\d t
+ \iint_{\Omega_T} u_k^{s}\left|\partial_t(\varphi^p)\right|\,\d x\,\d t\right)\\
& \leq c(p,s)\iint_{\Omega_T}u_k^{s}\left(|\nabla\varphi|^p+  \left|\partial_t(\varphi^p)\right|\right)\, \d x\,\d t\\
& \leq c(p,s)\iint_{\Omega_T}u^{s}\left(|\nabla\varphi|^p+\left|\partial_t(\varphi^p)\right| \right)\,\d x\,\d t<\infty,
\end{align*}
for every $k=2,3,\dots$.
By Lemma $\ref{sobolev}$, we have
\begin{align*}
\iint_{\Omega_T} \varphi^q u_k^{s-(2-p)+\frac{s p}{n}}\,\d x\,\d t 
&=\iint_{\Omega_T}\left|\varphi u_k^{\frac{s-(2-p)}{p}}\right|^q \,\d x\,\d t \\
&\leq c(p,s,n)\iint_{\Omega_T}\left|\nabla \left( \varphi u_k^{\frac{s-(2-p)}{p}} \right) \right|^p\,\d x\,\d t \left(\esssup_{0<t<T}\int_\Omega\varphi^m  u_k^s\,\d x\right)^{\frac{p}{n}},
\end{align*}
with $m=\frac{ps}{s-(2-p)}$ and $q=p+\frac{pm}{n}$.
Observe that 
\[
\ph(x,t)^m=\left(\ph(x,t)^{\frac{s}{s-(2-p)}}
\right)^p,
\] 
with $\frac mp > 1$ so that $\ph^{\frac mp} \in C^1_0(\Omega_T)$, and thus we may apply Lemma $\ref{unbdd_caccio}$ for second term as well.  We have
\begin{align*}
\esssup_{0<t<T}\int_\Omega\varphi^m  u_k^s\,\d x
& \leq c(p,s)\left(\iint_{\Omega_T}u_k^{s-(2-p)}|\nabla(\varphi^{\frac mp})|^p\,\d x\,\d t
+ \iint_{\Omega_T} u_k^{s}\left|\partial_t(\varphi^m)\right|\,\d x\,\d t\right)\\
&\leq c(p,s)\iint_{\Omega_T}u^{s}\left(|\nabla(\varphi^{\frac mp})|^p+\left|\partial_t(\varphi^{m})\right| \right)\,\d x\,\d t\\
&\le c(p,s)\iint_{\Omega_T}u^{s}\varphi^{m-p}\left(|\nabla\varphi|^p+\left|\partial_t(\varphi^p)\right| \right)\,\d x\,\d t<\infty
\end{align*}
for every $k=2,3,\dots$.
By combining the estimates and applying Lebesgue's monotone convergence theorem we conclude 
$$
\iint_{\Omega_T} \varphi^q u^{s-(2-p)+\frac{s p}{n}}\,\d x\,\d t
=\lim_{k\to\infty}\iint_{\Omega_T} \varphi^q u_k^{s-(2-p)+\frac{s p}{n}}\,\d x\,\d t 
<\infty.
$$
Since $s >\frac{n}{p}(2-p)$, it follows that $s-(2-p)+\frac{s p}{n}> s$. 
Thus $u\in L_{\loc}^{s(1+\frac{p}{n})-(2-p)}(\Omega_T)$.

Denote $s_0= s$ and $s_1=s_0(1+\frac{p}{n})-(2-p)$. 
After the first step of iteration we have $u\in L_{\loc}^{s_1}(\Omega_T)$.
If $s_1 <1$, then in the second step of iteration we choose $1-\epsilon= s_1$ 
and again combine the Caccioppoli inequality and the Sobolev inequality with $m = \frac{p s_1}{s_1-(2-p)}$. 
We continue in this way and obtain an increasing sequence of numbers $s_i$, satisfying 
\[
s_i = s_{i-1}\left(1+\tfrac{p}{n}\right) - (2-p),
\qquad i=1,2,\dots,
\]
which can be written in terms of $s_0$ as
\[
s_i =\left(1+\tfrac{p}{n}\right)^i\left(s_0-\tfrac{n}{p}(2-p)\right)+\tfrac{n}{p}(2-p),
\qquad i=0,1,2,\dots.
\]
After the $i$th step of iteration we have $u\in L_{\loc}^{s_i}(\Omega_T)$.
After a finite number of iterations we have $s_i\ge1$
and thus $u\in L_{\loc}^1(\Omega_T)$.

In order to pass from $u\in L_{\loc}^1(\Omega_T)$ to $u\in L_{\loc}^{p-1+\frac{p}{n}-\sigma} (\Omega_T)$ for every $\sigma>0$ we apply a similar argument once more.
Let $\eps = \frac{\sigma}{1+\frac pn}$. Then 
$p-1-\varepsilon+\tfrac{p(1-\varepsilon)}{n}
=p-1+\tfrac pn-\sigma$.
By Lemma $\ref{sobolev}$, we have
\begin{align*}
\iint_{\Omega_T} \varphi^q u_k^{p-1+\frac pn-\sigma}\,\d x\,\d t 
&=\iint_{\Omega_T} \varphi^q u_k^{p-1-\varepsilon+\frac{p(1-\varepsilon)}{n}}\,\d x\,\d t 
=\iint_{\Omega_T}\left|\varphi u_k^{\frac{p-1-\varepsilon}{p}}\right|^q \,\d x\,\d t \\
&\leq c(p,\varepsilon,n)\iint_{\Omega_T}\left|\nabla \left( \varphi u_k^{\frac{p-1-\varepsilon}{p}} \right) \right|^p\,\d x\,\d t 
\left(\esssup_{0<t<T}\int_\Omega\varphi^m  u_k^{1-\varepsilon}\,\d x\right)^{\frac{p}{n}},
\end{align*}
with $m=\frac{p(1-\eps)}{p-1-\eps}$ and $q=p+\frac{pm}{n}$.

By Lemma $\ref{unbdd_caccio}$, with $u_k\ge1$, this implies
\begin{equation}\label{e.gradient_est}
\begin{split}
\iint_{\Omega_T}&\left|\nabla\left(\varphi u_k^{\frac{p-1-\varepsilon}{p}}\right)\right|^p\,\d x\,\d t\\
& \leq c(p,\varepsilon)\left(\iint_{\Omega_T}u_k^{p-1-\varepsilon}|\nabla\varphi|^p\,\d x\,\d t
+ \iint_{\Omega_T} u_k^{-\varepsilon-1}|\nabla u_k|^p\varphi^p\,\d x\,\d t\right)\\
& \leq c(p,\varepsilon)\left(\iint_{\Omega_T}u_k^{p-1-\varepsilon}|\nabla\varphi|^p\,\d x\,\d t
+ \iint_{\Omega_T} u_k^{p-1-\varepsilon}\left|\partial_t(\varphi^p)\right|\,\d x\,\d t\right)\\
& \leq c(p,\varepsilon)\iint_{\Omega_T}u_k^{1-\varepsilon}\left(|\nabla\varphi|^p+  \left|\partial_t(\varphi^p)\right|\right)\, \d x\,\d t\\
& \leq c(p,\varepsilon)\iint_{\Omega_T}u\left(|\nabla\varphi|^p+\left|\partial_t(\varphi^p)\right|\right)\,\d x\,\d t<\infty
\end{split}
\end{equation}
for every $k=2,3,\dots$.
For the second term, Lemma $\ref{unbdd_caccio}$ implies
\begin{align*}
\esssup_{0<t<T}\int_\Omega\varphi^m  u_k^{1-\varepsilon}\,\d x
& \leq c(p,\varepsilon)\left(\iint_{\Omega_T}u_k^{p-1-\varepsilon}|\nabla(\varphi^{\frac mp})|^p\,\d x\,\d t
+ \iint_{\Omega_T} u_k^{p-1-\varepsilon}\left|\partial_t(\varphi^m)\right|\,\d x\,\d t\right)\\
&\leq c(p,\varepsilon)\iint_{\Omega_T}u^{1-\varepsilon}\left(|\nabla(\varphi^{\frac mp})|^p+\left|\partial_t(\varphi^{m})\right| \right)\,\d x\,\d t\\
&\le c(p,\varepsilon)\iint_{\Omega_T}u^{1-\varepsilon}\varphi^{m-p}\left(|\nabla\varphi|^p+\left|\partial_t(\varphi^p)\right| \right)\,\d x\,\d t\\
&\le c(p,\varepsilon)\iint_{\Omega_T}u\left(|\nabla\varphi|^p+\left|\partial_t(\varphi^p)\right| \right)\,\d x\,\d t
<\infty
\end{align*}
for every $k=2,3,\dots$.
Thus
\[
\iint_{\Omega_T} \varphi^q u^{p-1+\frac pn-\sigma}\,\d x\,\d t 
=\lim_{k\to\infty}\iint_{\Omega_T} \varphi^q u_k^{p-1+\frac pn-\sigma}\,\d x\,\d t<\infty,
\]
which implies $u\in L_{\loc}^{p-1+\frac{p}{n}-\sigma} (\Omega_T)$.
\end{proof}

Let $u$ be a nonnegative $p$-supercaloric function in $\Omega_T$.
It may happen that $u$ does not have a locally integrable weak derivative. 
In this case we consider the truncations $u_k=\min\{u, k\}\in L^p_{\loc}(\Omega_T)$, $k=1,2,\dots$ and define the weak gradient by
\begin{equation}\label{e.limgrad}
\nabla u = \lim_{k\to \infty} \nabla u_k,
\end{equation}
which is a well defined measurable function but does not necessarily belong to $L^1_{\loc}(\Omega_T)$.
If $\nabla u\in L^1_{\loc}(\Omega_T)$, then $\nabla u(\cdot,t)$ is the Sobolev gradient of $u(\cdot,t)$ for almost every $t$ with $0<t<T$.

\begin{theo} \label{thm:gradient_integbility}
Let $\frac{2n}{n+1}<p<2$ and let $\Omega$ be an open set in $\R^n$. Assume that $u$ is a $p$-supercaloric function in $\Omega_T$ with 
$u \in L_{\loc}^{s}(\Omega_T)$ for some $s > \frac{n}{p}(2-p)$, 
Then $\nabla u\in L^q_{\loc}(\Omega_T)$ whenever $0<q<p-1+\frac{1}{n+1}$. 
\end{theo}
\begin{proof}
Since $u$ is locally bounded from below, by adding a constant, we may assume that $u\ge1$. 
Let $0<t_1<t_2<T$, $\Omega' \Subset \Omega$ and $\eps \in (0,1)$.
We consider the truncations $u_k=\min\{u, k\}$, $k=1,2,\dots$. 
By H\"older's inequality, we have
\begin{align*}
\int_{t_1}^{t_2} \int_{\Omega' }& |\nabla u_k|^q \,\d x\,\d t 
= \int_{t_1}^{t_2} \int_{\Omega' } \left( u_k^{-\frac{1+\eps}{p}} |\nabla u_k| \right)^q u_k^{q\frac{1+\eps}{p}} \,\d x\,\d t \\
&\leq \left( \int_{t_1}^{t_2} \int_{\Omega' } u_k^{-1-\eps} |\nabla u_k|^p\, \d x \,\d t \right)^\frac{q}{p} 
\left( \int_{t_1}^{t_2} \int_{\Omega' } u_k^{q\frac{1+\eps}{p-q}}\,\d x\,\d t\right)^{1-\frac{q}{p}} \\
&= \left( \frac{p}{p-1-\eps} \right)^q \left( \int_{t_1}^{t_2} \int_{\Omega' } \left|\nabla \left( u_k^\frac{p-1-\eps}{p} \right) \right|^p \, \d x\, \d t \right)^\frac{q}{p} 
\left( \int_{t_1}^{t_2} \int_{\Omega' } u_k^{q\frac{1+\eps}{p-q}} \, \d x\, \d t \right)^{1-\frac{q}{p}},
\end{align*}
where the first integral is uniformly bounded with respect to $k$ as in \eqref{e.gradient_est} and the second integral is finite whenever
$q\frac{1+\eps}{p-q} <  p - 1 + \frac{p}{n}$
by Theorem~\ref{integbility_1}. From this we can conclude that the right hand side is finite for any $0<q < p - 1 + \frac{1}{n+1}$.
Now we have that $\nabla u_k$ is uniformly bounded in $L_{\loc}^q(\Omega_T)$. 
\end{proof}

\begin{rem}\label{r.limit_gradient}
If $p>\frac{2n+1}{n+1}$, then $p - 1 + \frac{1}{n+1}>1$ and
it follows that $\nabla u(\cdot,t)$ in Lemma \ref{thm:gradient_integbility} is the Sobolev gradient for almost every $t$ with $0<t<T$.
However, if $\frac{2n}{n+1} < p \leq\frac{2n+1}{n+1}$, then $p-1+\frac{1}{n+1} \leq1$. 
In this case the Sobolev gradient may not exist and the weak gradient $\nabla u$ is interpreted as in~\eqref{e.limgrad}.
\end{rem}

\begin{rem}
Let $\frac{2n}{n+1}<p<2$ and let $\Omega$ be an open set in $\R^n$. Assume that $u$ is a $p$-supercaloric function in $\Omega_T$ 
with $u \in L_{\loc}^{s}(\Omega_T)$ for some $s > \frac{n}{p}(2-p)$.
By Theorem \ref{thm:gradient_integbility} we have $\nabla u\in L^{p-1}_{\loc}(\Omega_T)$.
Theorem \ref{bdd_supersoln} implies
\begin{align*}
\iint_{\Omega_T}& \left(- u \partial_t \ph + |\nabla u|^{p-2}\nabla u \cdot \nabla \ph \right) \, \d x\, \d t \\
&=\lim_{k\to\infty}\iint_{\Omega_T} \left(- u_k \partial_t \ph + |\nabla u_k|^{p-2}\nabla u_k \cdot \nabla \ph \right) \, \d x\, \d t 
\geq 0
\end{align*}
for every nonnegative $\ph \in C_0^\infty(\Omega_T)$.
By the Riesz representation theorem there exists a nonnegative Radon measure $\mu$ on $\R^{n+1}$ such that
\[
\iint_{\Omega_T}\left(- u \partial_t \ph + |\nabla u|^{p-2}\nabla u \cdot \nabla \ph \right) \, \d x\, \d t 
=\iint_{\Omega_T}\varphi\,\d \mu 
\]
for every $\varphi\in C_0^\infty(\Omega_T)$. This means that $u$ is a solution to the measure data problem
\[
\partial_tu -\div\left( |\nabla u |^{p-2} \nabla u \right)=\mu.
\]
Note that $\nabla u\notin L^{p}_{\loc}(\Omega_T)$, in general. 
The integrability of the gradient of a weak solution to a measure data problem in the fast diffusion case has been studied Baroni \cite{Baroni2017}.
\end{rem}

\begin{rem}
Let $\frac{2n}{n+1}<p<2$ and let $\Omega$ be an open set in $\R^n$.  Assume that $u$ is a nonnegative $p$-supercaloric function in $\Omega_T$ 
with $u \in L_{\loc}^{2-p}(\Omega_T)$.
Then $\nabla \log u \in L^p_{\loc}(\Omega_T)$, where $\nabla$ stands for the usual Sobolev gradient. 
This property holds for $u$ satisfying the assumptions in Theorem \ref{integbility_1}, since $2-p < \frac{n}{p}(2-p)$. The logarithmic estimate can be obtained as in \cite{KinnunenLindqvist2005} by taking $\eps = p-1$ in Lemma~\ref{unbdd_caccio}. 
\end{rem}

\begin{rem}
By the following dichotomy it follows that $p$-supercaloric function is either locally integrable to any power smaller than $p -1 + \frac{p}{n} > 1$, or it is not integrable to a power $\frac{n}{p}(2-p) < 1$. Observe that as $p \searrow\frac{2n}{n+1}$ we have $\frac{n}{p}(2-p) \nearrow 1$ and $p -1 + \frac{p}{n} \searrow 1$. 
The gap between the exponents becomes smaller and shrinks to a point as $p$ approaches the critical value from above. 
\end{rem}

We have the following characterization for Barenblatt type $p$-supercaloric functions.

\begin{theo} \label{thm:barenblatt}
	Let $\frac{2n}{n+1}<p<2$ and let $\Omega$ be an open set in $\R^n$. Assume that $u$ is a  $p$-supercaloric function in $\Omega_T$. Then the following assertions are equivalent:
	\begin{itemize}
		\item[(i)] $u\in L^q_{\loc}(\Omega_T)$ for some $q> \frac{n}{p}(2-p)$,
		\item[(ii)] $u \in L^{\frac{n}{p}(2-p)}_{\loc} (\Omega_T)$,
		\item[(iii')] there exists $\alpha\in(\frac{n}{p}(2-p), 1)$ such that $$\esssup_{\delta<t<T-\delta} \int_{\Omega'} |u(x,t)|^\alpha\,\d x<\infty,$$
		whenever $\Omega'\Subset \Omega$ and $\delta\in(0, \frac{T}{2})$,
		\item[(iii)] 
		 $$\esssup_{\delta<t<T-\delta} \int_{\Omega'} |u(x,t)|\,\d x<\infty,$$
		whenever $\Omega'\Subset \Omega$ and $\delta\in(0, \frac{T}{2})$.
	\end{itemize}
\end{theo}
\begin{proof}
Since $u$ is locally bounded from below, by adding a constant, we may assume that $u\ge1$. 

(i) $\implies$ (ii): Direct consequence of H\"older's inequality. 

(iii) $\implies$ (iii'): H\"older's inequality. 
(iii') $\implies$ (i): Follows from the iteration argument in the proof of Theorem~\ref{integbility_1}. 

(i) $\implies$ (iii'): Let $\Omega'\Subset\Omega$ and $\delta\in(0, \frac{T}{2})$. 
We consider the truncations $u_k=\min\{u, k\}$, $k=2,3,\dots$. 
Since $u_k$ is a weak supersolution in $\Omega_T$, it satisfies the Caccioppoli inequality in Lemma $\ref{unbdd_caccio}$. 
Choose $\epsilon>0$ small enough such that $1-\epsilon = q$. By the assumption, $u\in L^{1-\epsilon}_{\loc}(\Omega_T)$. 
By choosing a test function $\varphi\in C_0^\infty(\Omega_T)$ with $0\leq \varphi \leq 1$ and $\varphi =1$ in $\Omega'\times (\delta, T-\delta)$ 
in Lemma \ref{unbdd_caccio}, we have
	\begin{align*}
	\esssup_{\delta<t<T-\delta}\int_{\Omega'} u_k^{1-\epsilon} \,\d x 
	&\leq c(p,\varepsilon)\left(\iint_{\Omega_T} u^{p-1-\epsilon}_k|\nabla\varphi|^p\,\d x\,\d t
	+\iint_{\Omega_T} u^{1-\epsilon}_k\left|\partial_t (\varphi^p)\right|\,\d x\,\d t\right)\\
	&\leq c(p,\varepsilon)\iint_{\Omega_T} u^{1-\epsilon}_k\left(|\nabla\varphi|^p+ \left|\partial_t (\varphi^p)\right|\right)\d x\,\d t\\
	&\leq c(p,\varepsilon)\iint_{\Omega_T} u^{1-\epsilon}\left(|\nabla\varphi|^p+ \left|\partial_t(\varphi^p)\right|\right)\d x\,\d t.
	\end{align*}
	The claim follows by letting $k\rightarrow\infty$. 
	
(ii) $\implies$ (i): Follows by showing contraposition  $\neg$(i) $\implies$ $\neg$(ii) by using Theorem~\ref{thm:monster}. 

(i) $\implies$ (iii): Follows by showing contraposition  $\neg$(iii) $\implies$ $\neg$(i) by using Theorem~\ref{thm:monster}.
\end{proof}

\begin{rem}
$\nabla u\in L^q_{\loc}(\Omega_T)$, whenever $0<q<p-1+\frac{1}{n+1}$, is another equivalent assertion in Theorem~\ref{thm:barenblatt} for $p >\frac{2n+1}{n+1}$.
This follows from the fact that $\nabla u$ is a weak gradient in Sobolev's sense and $u \in L^1_{\loc} (\Omega_T)$. 
By the proof of Theorem \ref{integbility_1}, this implies (i) in Theorem~\ref{thm:barenblatt}. 
\end{rem}

\section{Infinite point source solutions}

This section discusses $p$-supercaloric functions in the supercritical case that are not covered by Theorem \ref{thm:barenblatt}.
The leading example is the infinite point source solution in \eqref{razorblade}.
In contrast with the friendly giant in the slow diffusion case $p>2$, it is the singularity in space, not in time, that fails to be locally integrable to an appropriate power. 
Roughly speaking, at every time slice the singularity of the infinite point source solution is worse power type singularity than
\[
u(x,t)=|x|^{- \frac{n-p}{p-1}}, 
\quad 1<p<2,
\]
based on the fundamental solution to the elliptic $p$-Laplace equation.

\begin{example}
Let $\frac{2n}{n+1}<p<2$.
The infinite point source solution $U(x,t)$ in \eqref{razorblade} is a weak solution to $\eqref{evo_eqn}$ in $\R^n\times(0,\infty)$. 
It follows that the zero extension $u(x,t)$ as in \eqref{e.zero_barenblatt} is a $p$-supercaloric function in $\R^n\times\R$. 
However, $u(x,t)$ is not a weak supersolution to \eqref{evo_eqn} in any domain which contains a neighbourhood of $x = 0$,
since
\[
\int_{t_1}^{t_2}\int_{B(0,r)} |\nabla u|^{\frac{n}{2}(2-p)}\,\d x\,\d t=\infty,
\] 
for every $r>0$, $t_1>-\infty$ and $t_2>0$. 
This implies that $u\notin L_{\loc}^p(\R;W^{1,p}_{\loc}(\R^n))$.
Observe, that the truncations $\min\{u,k\}$, $k=1,2,\dots$ belong to $L_{\loc}^p(\R;W^{1,p}_{\loc}(\R^n))$ and are weak solutions to \eqref{evo_eqn}
by Theorem \ref{bdd_supersoln}.
\end{example}

\begin{example}
Let $\frac{2n}{n+1}<p<2$.
First show that there exist supercaloric functions that are not locally integrable to any positive power. 
Consider
$$
u(x,t)= \left( \frac{ct}{|x|^q} \right)^\frac{1}{2-p},
$$
where $q \ge p$ and the constant $c$ will be determined later. 
We show that $u$ is a supersolution to $\eqref{evo_eqn}$ in $(B(0,1) \setminus\{0\}) \times (0,\infty)$.
A direct computation shows that
\[
\partial_t u (x,t)= \tfrac{c^\frac{1}{2-p}}{2-p} t^\frac{p-1}{2-p} |x|^{-\frac{q}{2-p}}
\quad\text{and}\quad
\nabla u(x,t) = - \tfrac{q}{2-p} (ct)^\frac{1}{2-p} |x|^{-\frac{q}{2-p}-1} \frac{x}{|x|}.
\]
It follows that
\begin{align*}
\div\left( |\nabla u(x,t) |^{p-2} \nabla u(x,t) \right)= \left( \tfrac{q}{2-p} \right)^{p-1} \left(\tfrac{q}{2-p} -q + p -n\right) (ct)^\frac{p-1}{2-p} |x|^{q-p-\frac{q}{2-p}}.
\end{align*}
It is easy to see that
$\frac{q}{2-p} -q + p -n > 0$
for any $q \geq p$, and 
$\frac{q}{2-p} \geq -q + p + \frac{q}{2-p}$,
so that $|x|^{q-p-\frac{q}{2-p}} \leq |x|^{-\frac{q}{2-p}}$ for every $x\in B(0,1)$. By choosing
\begin{equation}\label{e.constant}
c=c(n,p,q)= (2-p) \left( \tfrac{q}{2-p} \right)^{p-1} \left( \tfrac{q}{2-p} - q + p -n \right),
\end{equation}
we have
\[
\partial_t u(x,t) \geq \div\left( |\nabla u(x,t) |^{p-2} \nabla u(x,t) \right)
\]
for every $(x,t)\in B(0,1) \times (0,\infty)$.
Lemma \ref{l.limit} implies that $u$ is a $p$-supercaloric function in $B(0,1) \times (0,\infty)$. 
For any $\eps > 0$, by choosing $q > \frac{n}{\eps}(2-p)$, we obtain a $p$-supercaloric function $u \notin L^{\eps}_{\loc}(B(0,1) \times (0,\infty))$.
This shows that, in general, a $p$-supercaloric function is not locally integrable to any positive power. 
The same example with the constant in \eqref{e.constant} is a supersolution to \eqref{evo_eqn} also for $1<p\leq \frac{2n}{n+1}$, with the additional requirement $q > \frac{(n-p)(2-p)}{p-1}$. 
The function $u$ is $p$-supercaloric for any $c > 0$ with $0<q \leq \frac{(n-p)(2-p)}{p-1}$ and $1<p<2$, since then
$$
\partial_t u(x,t) \geq 0 \geq \div\left( |\nabla u(x,t) |^{p-2} \nabla u(x,t) \right)
$$
for $(x,t) \in \R^n\times (0,\infty)$.
\end{example}

The following intrinsic weak Harnack inequality for supersolutions can be found in~\cite[Proposition 3.1]{GLL}.

\begin{lem} \label{lem:weak_harnack}
Let $\frac{2n}{n+1}<p<2$ and let $\Omega$ be an open set in $\R^n$.
Assume that $u$ is a nonnegative lower semicontinuous weak supersolution  to \eqref{evo_eqn} in $\Omega_T$. 
There exist constants $c_1=c_1(n,p) \in (0,1)$ and $c_2=c_2(n,p) \in (0,1)$ such that, for almost every $s \in (0,T)$, we have
$$
\inf_{B(x_0,2r)} u(\cdot,t) \geq c_1 \bint_{B(x_0,2r)}\, u(x,s)\, \d x
$$
for every $t \in[s+\frac{3}{4}\theta r^p, s+\theta r^p]$, with
$$
\theta= c_2\left( \bint_{B(x_0,2r)} u(x,s)\, \d x \right)^{2-p},
$$
whenever $B(x_0,16r) \times [s,s+\theta r^p] \subset \Omega_T$.
\end{lem}

We shall also apply the following Harnack inequality for weak solutions, see \cite[Appendix A, Proposition A.1.1]{digi}.

\begin{lem} \label{lem:L1_harnack}
Let $1<p<2$ and let $\Omega$ be an open set in $\R^n$.
Assume that $h$ is a nonnegative continuous weak solution to \eqref{evo_eqn} in $\Omega_T$. 
There exists a constant $c = c(p,n)$ such that 
$$
\sup_{s< \tau <t} \bint_{B(x_0,r)} h(x,\tau) \, \d x \leq c \inf_{s<\tau<t} \bint_{B(x_0,2r)} h(x,\tau) \, \d x 
+ c \left( \frac{t-s}{r^p} \right)^\frac{1}{2-p}
$$
whenever $B(x_0,2 r) \times [s,t] \subset \Omega_T$. 
\end{lem}

Next lemma will be a useful building block in the characterization of the complementary class of the Barenblatt type supercaloric functions.

\begin{lem} \label{lem:monster_harnack}
Let $\frac{2n}{n+1}<p<2$ and let $\Omega$ be an open set in $\R^n$.
Assume that $u$ is a nonnegative supercaloric function in $\Omega_T$ and let $(x_0,t_0) \in \Omega_T$.
Let $0<t_j<T$, $j=1,2,\dots$, with $t_j \to t_0$ as $j \to \infty$. 
If for every $r_0 > 0$ there exists $0<r \leq r_0$ such that 
\begin{equation} \label{eq:blowup-assumption}
\lim_{j\to \infty} \int_{B(x_0,r)} u(x,t_j) \, \d x = \infty,
\end{equation}
then 
$$
\liminf_{\substack{(x,t)\to (x_0,s) \\ t>s} } u(x,t)|x-x_0|^{\frac{p}{2-p}}>0 
$$
for every $s > t_0$.
\end{lem}

\begin{proof}
First we observe that~\eqref{eq:blowup-assumption} implies that 
\[
\lim_{j\to \infty} \int_{B(x_0,r)} u(x,t_j) \, \d x = \infty
\] 
for every $r > 0$. 
We construct a Poisson modification and apply Lemma~\ref{lem:L1_harnack} for the obtained solution and Lemma~\ref{lem:weak_harnack} for the truncations 
$u_k=\min\{u,k\}$, $k=1,2,\dots$, which are weak supersolutions by Theorem \ref{bdd_supersoln}.
\\
Let  $t \in (t_0,T)$ and $r > 0$.
The assumption in~\eqref{eq:blowup-assumption} implies that
$$
\bint_{B(x_0, r)} u(x,t_j) \, \d x \geq 4 c \left( \frac{t-t_j}{r^p} \right)^\frac{1}{2-p},
$$
for $j$ large enough. 
This implies that, when $j$ is large enough, we may choose a real number $k_j$ such that 
\begin{align} \label{eq:integral_level}
\bint_{B(x_0, r)} u_{k_j}(x,t_j) \, \d x = 2 c \left( \frac{t-t_j}{r^p} \right)^\frac{1}{2-p}.
\end{align}
We would like to obtain a weak solution $h_{k_j}$ to \eqref{evo_eqn} in $B(x_0,2r) \times (t_j,T)$ with lateral and initial boundary values $u_{k_j}$. 
We will do this via approximation in order to guarantee existence of solution to such a problem. 
Since $u_{k_j}$ is lower semicontinuous, there exists a sequence of continuous functions $\psi_{k_j,i}$, $i=1,2,\dots$, 
with $0 \leq \psi_{k_j,i}\leq \psi_{k_j,i+1} \leq u_{k_j}$, $i=1,2,\dots$, and 
$\psi_{k_j,i} \nearrow u_{k_j}$ pointwise in $\Omega_T$ as $i\to \infty$. 
By Theorem \ref{t.existence}, there exists a weak solution $h_{k_j,i}\in C(\overline{B(x_0,2r) \times (t_j,T)})$ to \eqref{evo_eqn} in $B(x_0,2r) \times (t_j,T)$ 
with $h_{k_j,i}=\psi_{k_j,i}$ on $\partial_p(B(x_0,2r) \times (t_j,T))$. 
By the parabolic comparison principle in Theorem \ref{t.parcoparison}, we have $h_{k_j,i} \leq u_{k_j}$ in $B(x_0,2r) \times (t_j,T)$
for every $i=1,2,\dots$.  Lemma~\ref{lem:L1_harnack}, with $s = t_j$ and $t < T$, implies
\begin{align*}
\sup_{t_j< \tau <t} \bint_{B(x_0,r)} h_{k_j,i}(x,\tau) \, \d x &\leq c \inf_{t_j<\tau<t} \bint_{B(x_0,2r)} h_{k_j,i}(x,\tau) \, \d x 
+ c \left( \frac{t-t_j}{r^p} \right)^\frac{1}{2-p} \\
&\leq c \bint_{B(x_0,2r)} u_{k_j}(x,\tau) \, \d x + c \left( \frac{t-t_j}{r^p} \right)^\frac{1}{2-p}
\end{align*}
for every $\tau \in (t_j, t)$.
On the other hand, we have
$$
\sup_{t_j< \tau <t} \bint_{B(x_0,r)} h_{k_j,i}(x,\tau) \, \d x \geq\bint_{B(x_0,r)} h_{k_j,i}(x,t_j) \, \d x = \bint_{B(x_0,r)} \psi_{k_j,i}(x,t_j) \, \d x.
$$
By combining the two inequalities above, we have
\[
\bint_{B(x_0,r)} \psi_{k_j,i}(x,t_j) \, \d x
\leq c \bint_{B(x_0,2r)} u_{k_j}(x,\tau) \, \d x +c \left( \frac{t-t_j}{r^p} \right)^\frac{1}{2-p}.
\]
Passing to the limit $i\to \infty$ and using~\eqref{eq:integral_level}, we obtain
\begin{align*}
2 c \left( \frac{t-t_j}{r^p} \right)^\frac{1}{2-p} 
&=\bint_{B(x_0,2r)} u_{k_j}(x,t_j) \, \d x
=\lim_{i\to\infty}\bint_{B(x_0,r)} \psi_{k_j,i}(x,t_j) \, \d x\\
&\le c \bint_{B(x_0,2r)} u(x,\tau) \, \d x 
+ c \left( \frac{t-t_j}{r^p} \right)^\frac{1}{2-p}
\end{align*}
and thus
\[
\left( \frac{t-t_j}{r^p} \right)^\frac{1}{2-p} 
\le \bint_{B(x_0,2r)} u(x,\tau) \, \d x
\]
for every $\tau \in (t_j,t)$. By passing $j\to \infty$, we obtain 
\begin{align} \label{eq:lowerbound}
r^p \left( \bint_{B(x_0,2r)} u(x,\tau)\, \d x \right)^{2-p} 
\geq t-t_0
\end{align}
for every $\tau \in (t_0,t)$. Notice that the construction above can be done for any $r \in (0,r_0]$, which implies that~\eqref{eq:lowerbound} holds for every $\tau \in (t_0,t)$ and $r \in (0,r_0]$.

Let $(t_j)_{j\in\N}$ be a sequence for which~\eqref{eq:blowup-assumption} holds and let $r_j \to 0$ as $j\to \infty$.
It follows that
\begin{equation} \label{eq:liminf_lowerbound}
\liminf_{j\to \infty } r_j^p \left( \bint_{B(x_0,2r_j)} u(x,t_j)\, \d x \right)^{2-p} 
\geq  t-t_0 >0.
\end{equation}
This implies that
\[
\lim_{j\to\infty}\bint_{B(x_0,2r_j)} u(x,t_j)\, \d x=\infty.
\]
Let $\eps\in[0, c_2(t-t_0))$ and, for every $j$ large enough, choose the truncation levels $k_j$ so that
$$
\theta_j r_j^p = c_2 r_j^p \left( \bint_{B(x_0,2r_j)} u_{k_j}(x,t_j) \, \d x  \right)^{2-p} = \eps.
$$
The constant $c_2 = c_2(n,p) \in (0,1)$ is from Lemma~\ref{lem:weak_harnack}. Since $u_{k_j}$ is a nonnegative weak supersolution to \eqref{evo_eqn} in $\Omega_T$, 
the weak Harnack inequality in Lemma~\ref{lem:weak_harnack} implies
\[
\inf_{B(x_0,2r_j)} u(\cdot,t) \geq \inf_{B(x_0,2r_j)} u_{k_j}(\cdot,t) 
\geq c(n,p) \bint_{B(x_0,2r_j)} u_{k_j}(x,t_j)\, \d x 
\geq c(n,p)\eps^{\frac1{2-p}}r_j^{-\frac{p}{2-p}}
\]
for every $t \in[ t_j + \frac{3}{4}\eps, t_j + \eps]$ and $j$ large enough. 
Since sequence $r_j \to 0$ is arbitrary, the claim follows.
\end{proof}

Next we give a characterization of the complementary class of the Barenblatt type supercaloric functions.

\begin{theo} \label{thm:monster}
	Let $\frac{2n}{n+1}<p<2$  and let $\Omega$ be an open set in $\R^n$. Assume that $u$ is a $p$-supercaloric function in $\Omega_T$. Then the following properties are equivalent:
	\begin{itemize}
		\item[(i)] $u\notin L^q_{\loc}(\Omega_T)$ for any $q> \frac{n}{p}(2-p)$,
		\item[(ii)] $u \notin L^{\frac{n}{p}(2-p)}_{\loc} (\Omega_T)$,
		\item[(iii)] there exists $\Omega'\Subset \Omega$ and $\delta\in(0, \frac{T}{2})$ such that 
		$$\esssup_{\delta<t<T-\delta} \int_{\Omega'} |u(x,t)|\, \d x=\infty,$$
		\item[(iii')] for every $\alpha\in(\frac{n}{p}(2-p), 1)$ there exists $\Omega'\Subset \Omega$ and $\delta\in(0, \frac{T}{2})$ such that 
		$$\esssup_{\delta<t<T-\delta} \int_{\Omega'} |u(x,t)|^\alpha\,\d x=\infty.$$	
		\item[(iv)] there exists $(x_0,t_0) \in \Omega_T$ such that
		$$ \liminf_{\substack{(x,t)\to (x_0,s) \\ t>s} } u(x,t)|x-x_0|^{\frac{p}{2-p}}>0 $$
		for every $s>t_0$.
	\end{itemize}
\end{theo}
\begin{proof} Since $u$ is locally bounded from below, by adding a constant, we may assume that $u\ge1$. 

(ii) $\implies$ (i): H\"older's inequality. 
(i) $\iff$ (iii'): Direct consequence of Theorem~\ref{thm:barenblatt}. 

(iii') $\implies$ (iii): H\"older's inequality. 

(iv) $\implies$ (iii'): There exists $r>0$ and $\delta>0$ such that 
$$
u(x,t) |x-x_0|^\frac{p}{2-p} \geq \eps> 0
$$
for every $(x,t) \in \left( B_r(x_0)\setminus \{x_0\} \right) \times (t_0, t_0+\delta)$. This implies (iii').  

(iv) $\implies$ (ii): The same argument as previous implication. 

(iii) $\implies$ (iv): This is the implication which requires a bit more machinery, especially use of Harnack inequalities. 
From (iii) we have that there exists a point $t_0 \in (0,T)$ and a sequence $t_j \to t_0$ as $j \to \infty$ such that
$$
\lim_{j\to \infty} \int_{\Omega'}u(x,t_j)\, \d x = \infty.
$$

The collection $\{B(x,r_x): x\in \overline{\Omega'},r_x>0\}$ is an open cover of $\overline{\Omega'}$. 
Since $\overline{\Omega'}$ is compact, there exists a finite subcover 
$\{B(x_i,r_i):i=1,\dots,N\}$ and
$$
\int_{\Omega'} u(x,t_j)\, \d x \leq \sum_{i=1}^N \int_{B(x_i,r_i)} u(x,t_j)\, \d x
$$
for every $j=1,2,\dots$. By taking the limit $j \to \infty$ it follows that
\begin{align} \label{eq:ball_blowup}
\lim_{j \to \infty} \int_{B(x_i,r_i)} u(x,t_j) \, \d x = \infty
\end{align}
for some $i\in \{1,\dots,N\}$. We claim that there exists $x_0 \in B(x_i,r_i)$ such that 
\begin{align} \label{eq:smallball_blowup}
\lim_{j \to \infty} \int_{B(x_0,r)} u(x,t_j) \, \d x = \infty
\end{align}
for arbitrarily small $r>0$.
For a contradiction, assume that for every point $y\in  B(x_i,r_i)$ there exists $r_y>0$ such that
$$
\limsup_{j\to \infty} \int_{B(y,r)} u(x,t_j)\, \d x < \infty
$$
for every $0<r \leq r_y$. 
The collection $\{B(y,r_y):y\in \overline{B}(x_i,r_i)\}$ is an open cover of  $\overline{B}(x_i,r_i)$ and 
thus there exists a finite subcover $\{ B(y_k, r_k):k=1,2,\dots,M\}$. Since
$$
\int_{B(x_i,r_i)} u(x,t_j) \, \d x \leq \sum_{k=1}^M \int_{B(y_k,r_k)} u(x,t_j) \, \d x,
$$
we obtain a contradiction by taking the limit $j \to \infty$ as the left hand side goes to infinity and right hand side does not. Thus~\eqref{eq:smallball_blowup} holds. Now the assumption for Lemma~\ref{lem:monster_harnack} holds which completes the proof. 
\end{proof}

\begin{rem} \label{rem:endpoint_integbility}
We complete the proof of Theorem~\ref{thm:barenblatt} by showing the remaining implication (ii) $\implies$ (i) by contraposition. 
Assume that the assertion (i) does not hold in Theorem~\ref{thm:barenblatt}. 
Then  $u\notin L^q_{\loc}(\Omega_T)$ for any $q> \frac{n}{p}(2-p)$ and Theorem~\ref{thm:monster} implies that $u \notin L^{\frac{n}{p}(2-p)}_{\loc} (\Omega_T)$. 
Also the assertion $u \in L^\infty_{\loc} ( 0,T; L^1_{\loc}(\Omega) )$ in Theorem~\ref{thm:barenblatt} follows similarly from Theorem~\ref{thm:monster}.
\end{rem}

\end{document}